\documentclass[a4paper,oneside,10pt]{article}

\usepackage[english]{babel}
\usepackage[latin1]{inputenc}
\usepackage[T1]{fontenc}
\usepackage{lmodern}
\usepackage{cite}
\usepackage{microtype}
\usepackage{amssymb}
\usepackage{amsmath}
\usepackage{mathrsfs}
\usepackage{dsfont}

\newcommand{\ind}[1]{\mathds{1}_{#1}}
\newcommand{\R}{\mathbb{R}}
\newcommand{\E}{\mathbb{E}}
\newcommand{\G}{\Gamma}
\renewcommand{\d}{\delta}
\newcommand{\eps}{\varepsilon}
\renewcommand{\r}{\rho}
\newcommand{\s}{\sigma}

\usepackage{amsthm}
\newtheorem{theo}{Theorem}
\newtheorem{lem}[theo]{Lemma}
\renewenvironment{proof}{\noindent{\bf Proof.}}{\qed}

\begin{document}

\renewcommand{\contentsname}{Contents}
\renewcommand{\refname}{\textbf{References}}
\renewcommand{\abstractname}{Abstract}

\begin{center}
{\bf \begin{Large}
Fluctuations of the Self-Normalized Sum\medskip

in the Curie-Weiss Model of SOC
\end{Large}}\bigskip \bigskip\bigskip\bigskip

\begin{large}
{\setlength{\tabcolsep}{15pt}
\begin{tabular}{p{3.8cm}p{1cm}p{3.8cm}}
\centering Matthias Gorny & \centering and &\centering S. R. S. Varadhan\tabularnewline
\centering{\it Universit\'e Paris Sud}
\centering & & \centering{\it Courant Institute}\tabularnewline
\centering{\it and ENS Paris}&&\centering{\it New York}
\end{tabular}}
\end{large}
\end{center}\bigskip \bigskip\bigskip

\begin{abstract}
\noindent We extend the main theorem of~\cite{CerfGorny} about the fluctuations in the Curie-Weiss model of SOC. We present a short proof using the Hubbard-Stratonovich transformation with the self-normalized sum of the random variables. 
\end{abstract}
\bigskip \bigskip \bigskip\bigskip



\bigskip\bigskip \bigskip
\section{Introduction}

In~\cite{CerfGorny}, Rapha\"el Cerf and Matthias Gorny designed a Curie-Weiss model of self-organized criticality. It is the model given by an infinite triangular array of real-valued random variables $(X_{n}^{k})_{1\leq k \leq n}$ such that for all $n \geq 1$, $(X^{1}_{n},\dots,X^{n}_{n})$ has the distribution
\[d\widetilde{\mu}_{n,\r}(x_{1},\dots,x_{n})=\frac{1}{Z_{n}}\exp\left(\frac{1}{2}\frac{(x_{1}+\dots+x_{n})^{2}}{x_{1}^{2}+\dots+x_{n}^{2}}\right)\ind{\{x_{1}^{2}+\dots+x_{n}^{2}>0\}}\,\prod_{i=1}^{n}d\r(x_{i}),\]
where $\r$ is a probability measure on $\R$ which is not the Dirac mass at~$0$, and where $Z_n$ is the normalization constant. This model is a modification of the generalized Ising Curie-Weiss model by the implementation of an automatic control of the inverse temperature. 

For any $n\geq 1$, we denote
\[S_{n}=X^{1}_{n}+\dots+X^{n}_{n},\qquad T_{n}=(X^{1}_{n})^{2}+\dots+(X^{n}_{n})^{2}.\]
By using Cram\'er's theory and Laplace's method, Cerf and Gorny proved in~\cite{CerfGorny} that, if $\r$ satisfies
\begin{equation}
\exists v_0>0\qquad \int_{\R}e^{v_0 z^2}\,d\r(z)<+\infty
\tag{$*$}
\end{equation}
and if $\r$ has a bounded density, then
\[\frac{S_n}{n^{3/4}}\overset{\mathscr{L}}{\underset{n \to \infty}{\longrightarrow}} \left(\frac{4\mu_{4}}{3\s^{8}}\right)^{1/4}\G\left(\frac{1}{4}\right)^{-1} \exp\left(-\frac{\mu_{4}}{12 \s^{8}}s^{4}\right)\,ds.\]
The case where $\r$ is a centered Gaussian measure has been studied in~\cite{GORGaussCase}. This fluctuation result shows that this model is a self-organized model exhibiting critical behaviour. Indeed it has the same behaviour as the critical generalized Ising Curie-Weiss model (see~\cite{EN}) and, by construction, it does not depend on any external parameter.

This result has been extended in~\cite{Gorny3} to the case where $\r$ satisfies some Cram\'er condition, which is fulfilled in particular when $\r$ has a an absolutely continuous component. However the proof is very technical and it does not deal with the case where $\r$ is discrete for example.

In this paper we prove that the convergence in distribution of $S_n/n^{3/4}$, under $\smash{\widetilde{\mu}_{n,\r}}$, is true for any symmetric probability measure $\r$ on $\R$ which satisfies~$(*)$. To this end, we study the fluctuations of the self-normalized sum $S_n/\sqrt{T_n}$. With this term, it is possible to use the so-called Hubbard-Stratonovich transformation as in lemma~3.3 of~\cite{EN}, which is the key ingredient for the proof of the fluctuations theorem in the generalized Ising Curie-Weiss model.

\begin{theo} Let $\r$ be a symmetric probability measure on $\R$ which is not the Dirac mass at~$0$ and which has a finite fifth moment. We denote by $\s^2$ the variance of~$\r$ and by $\mu_4$ its fourth moment. Then, under $\smash{\widetilde{\mu}_{n,\r}}$,
\[\frac{S_{n}}{n^{1/4}\sqrt{T_n}} \overset{\mathscr{L}}{\underset{n \to \infty}{\longrightarrow}} \left(\frac{4\mu_{4}}{3\s^{4}}\right)^{1/4}\G\left(\frac{1}{4}\right)^{-1} \exp\left(-\frac{\mu_{4}}{12 \s^{4}}s^{4}\right)\,ds.\]
\label{Maintheo}
\end{theo}

Remark: the hypothesis that $\r$ has a fifth moment may certainly be weakened by assuming instead that
\[\exists \eps>0\qquad \int_{\R}|z|^{4+\eps}\,d\r(z)<+\infty.\]

We prove theorem~\ref{Maintheo} in section~\ref{Proof}. If we add the hypothesis that $\r$ satisfies~$(*)$ then, under $\smash{\widetilde{\mu}_{n,\r}}$, $T_n/n$ converges in probability to $\s^2$. This result is proved in section~3 of~\cite{Gorny3} using Cram\'er's theorem, Varadhan's lemma (see~\cite{DZ}) and a conditioning argument. Moreover
\[\forall n\geq 1\qquad \frac{S_n}{n^{3/4}}=\sqrt{\frac{T_n}{n}}\times\frac{S_{n}}{n^{1/4}\sqrt{T_n}},\]
and condition $(*)$ implies that $\r$ has finite moments of all orders. Therefore the following theorem is a consequence of theorem~\ref{Maintheo} and Slutsky lemma (theorem~3.9~of~\cite{BillCVPM}).

\begin{theo} Let $\r$ be a symmetric probability measure on $\R$ which is not the Dirac mass at~$0$ and such that
\[\exists v_0>0\qquad \int_{\R}e^{v_0 z^2}\,d\r(z)<+\infty.\]
Then, under $\smash{\widetilde{\mu}_{n,\r}}$,
\[\frac{S_n}{n^{3/4}}\overset{\mathscr{L}}{\underset{n \to \infty}{\longrightarrow}} \left(\frac{4\mu_{4}}{3\s^{8}}\right)^{1/4}\G\left(\frac{1}{4}\right)^{-1} \exp\left(-\frac{\mu_{4}}{12 \s^{8}}s^{4}\right)\,ds.\]
\end{theo}\pagebreak

\section{Proof of theorem~\ref{Maintheo}}
\label{Proof}

Let $(X_{n}^{k})_{1\leq k \leq n}$ be an infinite triangular array of random variables such that, for any $n\geq 1$, $(X_n^1,\dots,X_n^n)$ has the law $\smash{\widetilde{\mu}_{n,\r}}$. Let us recall that
\[\forall n\geq 1\qquad S_n=X_n^1+\dots+X_n^n\qquad\mbox{and}\qquad T_n=(X_n^1)^2+\dots+(X_n^n)^2,\]
and that $T_n>0$ almost surely. We use the Hubbard-Stratonovich transformation: let $W$ be a random variable with standard normal distribution and which is independent of $(X_{n}^{k})_{1\leq k \leq n}$. Let $n\geq 1$ and let $f$ be a bounded continuous function on $\R$. We put
\[E_n=\E\left[f\left(\frac{W}{n^{1/4}}+\frac{S_{n}}{n^{1/4}\sqrt{T_n}}\right)\right].\]
We introduce $(Y_i)_{i\geq 1}$ a sequence of independent random variables with common distribution $\r$. We have
\begin{multline*}
E_n=\frac{1}{Z_n\sqrt{2\pi}}\,\E\Bigg[\int_{\R}f\left(\frac{w}{n^{1/4}}+\frac{Y_1+\dots+Y_n}{n^{1/4}\sqrt{Y_1^2+\dots+Y_n^2}}\right)\\
\times\exp\left(\frac{1}{2}\frac{(Y_{1}+\dots+Y_{n})^{2}}{Y_{1}^{2}+\dots+Y_{n}^{2}}-\frac{w^2}{2}\right)\ind{\{Y_{1}^{2}+\dots+Y_{n}^{2}>0\}}\,dw\Bigg].
\end{multline*}
We make the change of variable
\[z=\frac{w}{n^{1/4}}+\frac{Y_1+\dots+Y_n}{n^{1/4}\sqrt{Y_1^2+\dots+Y_n^2}}\]
in the integral and we get
\begin{multline*}
E_n=\frac{n^{1/4}}{Z_n\sqrt{2\pi}}\,\E\Bigg[\ind{\{Y_{1}^{2}+\dots+Y_{n}^{2}>0\}}\\
\times\int_{\R}f\left(z\right)\exp\left(-\frac{\sqrt{n}z^2}{2}+zn^{1/4}\frac{Y_1+\dots+Y_n}{\sqrt{Y_1^2+\dots+Y_n^2}}\right)\,dz\Bigg].
\end{multline*}
Let $U_1,\dots,U_n,\eps_1,\dots,\eps_n$ be independent random variables such that the distribution of $U_i$ is~$\r$ and the distribution of $\eps_i$ is $(\d_{-1}+\d_1)/2$, for any $i\in \{1,\dots,n\}$. Since $\r$ is symmetric, the random variables $\eps_1U_1,\dots,\eps_nU_n$ are also independent with common distribution $\r$. As a consequence
\begin{multline*}
E_n=\frac{n^{1/4}}{Z_n\sqrt{2\pi}}\,\E\Bigg[\ind{\{U_{1}^{2}+\dots+U_{n}^{2}>0\}}\\
\times\int_{\R}f\left(z\right)\exp\left(-\frac{\sqrt{n}z^2}{2}+\sum_{i=1}^n\frac{zn^{1/4}\eps_iU_i}{\sqrt{U_1^2+\dots+U_n^2}}\right)\,dz\Bigg].
\end{multline*}
For any $i\in\{1,\dots,n\}$, we denote (in the case where $U_{1}^{2}+\dots+U_{n}^{2}>0$)
\[A_{i,n}=\frac{U_i}{\sqrt{U_1^2+\dots+U_n^2}}.\]
By using Fubini's theorem and the independence of $\eps_i,U_i$, $i\in \{1,\dots,n\}$, we obtain
\begin{align*}
E_n&=\frac{n^{1/4}}{Z_n\sqrt{2\pi}}\,\E\Bigg[\ind{\{U_{1}^{2}+\dots+U_{n}^{2}>0\}}\,\int_{\R}f(z)\exp\left(-\frac{\sqrt{n}z^2}{2}\right)\\
&\qquad\qquad\qquad\qquad\qquad\times \E\left(\,\prod_{i=1}^n\exp\left(zn^{1/4}\eps_iA_{i,n}\right)\,\Bigg|\,(U_1,\dots,U_n)\,\right)\,dz\Bigg].\\
&=\frac{n^{1/4}}{Z_n\sqrt{2\pi}}\,\E\Bigg[\ind{\{U_{1}^{2}+\dots+U_{n}^{2}>0\}}\,\int_{\R}f\left(z\right)\exp\left(-\frac{\sqrt{n}z^2}{2}\right)\\
&\qquad\qquad\qquad\qquad\qquad\times \exp\left(\sum_{i=1}^n\ln\mathrm{cosh}\,(zn^{1/4}A_{i,n})\right)\,dz\Bigg].
\end{align*}
We define the function $g$ by
\[\forall y\in \R\qquad g(y)=\ln\mathrm{cosh}\,y-\frac{y^2}{2}.\]
It is easy to see that $g(y)< 0$ for $y>0$.
We notice that $A_{1,n}^2+\dots+A_{n,n}^2=1$, so that
\[E_n=\frac{n^{1/4}}{Z_n\sqrt{2\pi}}\,\E\left[\ind{\{U_{1}^{2}+\dots+U_{n}^{2}>0\}}\,\int_{\R}f\left(z\right)\exp\left(\sum_{i=1}^n g(zn^{1/4}A_{i,n})\right)\,dz\right].\]
Now we use Laplace's method. Let us examine the convergence of the term in the exponential: for any $i\in \{1,\dots,n\}$, the Taylor-Lagrange formula states that there exists a random variable $\xi_i$ such that
\[g(zn^{1/4}A_{i,n})=-\frac{(zn^{1/4}A_{i,n})^4}{12}+\frac{(zn^{1/4}A_{i,n})^5}{5!}g^{(5)}(\xi_i).\]
By a simple computation, we see that the function $g^{(5)}$ is bounded over $\R$. As a consequence
\begin{multline*}
\sum_{i=1}^n g(zn^{1/4}A_{i,n})=-\frac{z^4}{12}\,\frac{(Y_1^4+\dots+Y_n^4)/n}{((Y_1^2+\dots+Y_n^2)/n)^2}\\
+z^5\,\frac{(Y_1^5+\dots+Y_n^5)/n}{((Y_1^2+\dots+Y_n^2)/n)^{5/2}}\,O\left(\frac{1}{n^{1/4}}\right).
\end{multline*}
By hypothesis, the distribution $\r$ has a finite fifth moment. Hence the law of large numbers implies that
\[\forall z\in \R\qquad \sum_{i=1}^n g(zn^{1/4}A_{i,n})\underset{n\to+\infty}{\longrightarrow}-\frac{\mu_4 z^4}{12\s^4}\qquad \mbox{a.s}.\]

\begin{lem} There exists $c>0$ such that
\[\forall z\in \R\quad \forall n\geq 1\qquad \sum_{i=1}^n g(zn^{1/4}A_{i,n})\leq -\frac{cz^4}{1+z^2/\sqrt{n}}.\]
\end{lem}

\begin{proof} We define $h$ by
\[\forall y\in \R\backslash\{0\}\qquad h(y)=\frac{1+y^2}{y^4}g(y).\]
It is a negative continuous function on $\R\backslash\{0\}$. Since $g(y)\sim -y^4/12$ in the neighbourhood of~$0$, the function $h$ can be extended to a   function continuous on~$\R$ by putting $h(0)=-1/12$. Next we have
\[\forall y\in \R\backslash\{0\}\qquad h(y)=\frac{1+y^2}{y^2}\times \left(\frac{\ln\mathrm{cosh}\,y}{y^2}-\frac{1}{2}\right),\]
so that $h(y)$ goes to $-1/2$ when $|y|$ goes to $+\infty$.
Therefore $h$ is bounded by some constant $-c$ with $c>0$. Next we easily check that $x\longmapsto x^2/(1+x)$ is convex on $[0,+\infty[$ so that, for any $z\in \R$ and $n\geq 1$,
\begin{align*}
\sum_{i=1}^n g(zn^{1/4}A_{i,n})&\leq -nc\,\frac{1}{n}\sum_{i=1}^n\frac{(zn^{1/4}A_{i,n})^4}{1+(zn^{1/4}A_{i,n})^2}\\
&\leq -nc\,\frac{\left(\frac{1}{n}\sum_{i=1}^n(zn^{1/4}A_{i,n})^2\right)^2}{1+\frac{1}{n}\sum_{i=1}^n(zn^{1/4}A_{i,n})^2}=-\frac{cz^4}{1+z^2/\sqrt{n}},
\end{align*}
since $A_{1,n}^2+\dots+A_{n,n}^2=1$.
\end{proof}
\bigskip

\noindent If $|z|\leq n^{1/4}$ then $1+z^2/\sqrt{n}\leq 2$ and thus, by the previous lemma,
\[\left|\ind{\{U_{1}^{2}+\dots+U_{n}^{2}>0\}}\,\ind{|z|\leq n^{1/4}}\,\exp\left(\sum_{i=1}^n g(zn^{1/4}A_{i,n})\right)\right|\leq \exp\left(-\frac{cz^4}{2}\right).\]
Since
\[\E\left[\int_{\R}\left\|f\right\|_{\infty}\exp\left(-\frac{cz^4}{2}\right)\,dz\right]<+\infty,\]
the dominated convergence theorem implies that
\begin{multline*}
\E\left[\ind{\{U_{1}^{2}+\dots+U_{n}^{2}>0\}}\,\int_{\R}\ind{|z|\leq n^{1/4}}\,f\left(z\right)\exp\left(\sum_{i=1}^n g(zn^{1/4}A_{i,n})\right)\,dz\right]\\
\underset{n\to+\infty}{\longrightarrow}\!\int_{\R}f(z)\exp\left(-\frac{\mu_4 z^4}{12\s^4}\right)\,dz.
\end{multline*}
If $|z|> n^{1/4}$ then $1+z^2/\sqrt{n}\leq 2z^2/\sqrt{n}$ and thus, by the previous lemma,
\[\left|\ind{\{U_{1}^{2}+\dots+U_{n}^{2}>0\}}\,\ind{|z|> n^{1/4}}\,\exp\left(\sum_{i=1}^n g(zn^{1/4}A_{i,n})\right)\right|\leq \exp\left(-\frac{c\sqrt{n}z^2}{2}\right).\]
Hence
\[\E\left[\ind{\{U_{1}^{2}+\dots+U_{n}^{2}>0\}}\,\int_{\R}\ind{|z|> n^{1/4}}\,f\left(z\right)\exp\left(\sum_{i=1}^n g(zn^{1/4}A_{i,n})\right)\,dz\right]\leq\frac{\left\|f\right\|_{\infty}\sqrt{2\pi}}{n^{1/4}\sqrt{c}},\]
and thus
\begin{multline*}
\E\left[\ind{\{U_{1}^{2}+\dots+U_{n}^{2}>0\}}\,\int_{\R}f\left(z\right)\exp\left(\sum_{i=1}^n g(zn^{1/4}A_{i,n})\right)\,dz\right]\\
\underset{n\to+\infty}{\longrightarrow}\int_{\R}f(z)\exp\left(-\frac{\mu_4 z^2}{12\s^4}\right)\,dz.
\end{multline*}
If we take $f=1$, we get
\[\frac{Z_n\sqrt{2\pi}}{n^{1/4}}\underset{n\to+\infty}{\longrightarrow}\int_{\R}\exp\left(-\frac{\mu_4 z^4}{12\s^4}\right)\,dz.\]
We have proved that
\[\frac{W}{n^{1/4}}+\frac{S_{n}}{n^{1/4}\sqrt{T_n}} \overset{\mathscr{L}}{\underset{n \to \infty}{\longrightarrow}} \left(\int_{\R}\exp\left(-\frac{\mu_4 z^4}{12\s^4}\right)\,dz\right)^{-1} \exp\left(-\frac{\mu_{4}}{12 \s^{4}}s^{4}\right)\,ds.\]
Since $(n^{-1/4}W)_{n\geq 1}$ converges in distribution to $0$, Slutsky lemma (theorem~3.9 of~\cite{BillCVPM}) implies that
\[\frac{S_{n}}{n^{1/4}\sqrt{T_n}} \overset{\mathscr{L}}{\underset{n \to \infty}{\longrightarrow}} \left(\int_{\R}\exp\left(-\frac{\mu_4 z^4}{12\s^4}\right)\,dz\right)^{-1} \exp\left(-\frac{\mu_{4}}{12 \s^{4}}s^{4}\right)\,ds.\]
By an ultimate change of variables we compute that
\[\int_{\R}\exp\left(-\frac{\mu_4 z^4}{12\s^4}\right)\,dz=\left(\frac{3\s^{4}}{4\mu_{4}}\right)^{1/4}\G\left(\frac{1}{4}\right).\] 
This ends the proof of theorem~\ref{Maintheo}.\bigskip

\noindent {\bf{Acknowledgement.}}  The second author was supported
partially by NSF grant DMS 1208334.

\end{document}